\documentclass[11pt]{amsart}
\usepackage{amsmath}
\usepackage[active]{srcltx}
\usepackage{t1enc}
\usepackage[latin2]{inputenc}
\usepackage{verbatim}
\usepackage{amsmath,amsfonts,amssymb,amsthm}
\usepackage[mathcal]{eucal}
\usepackage{enumerate}
\usepackage[centertags]{amsmath}
\usepackage{graphics}

\newtheorem{theorem}{Theorem}

\newtheorem{lemma}{Lemma}
\newtheorem{remark}{Remark}
\newtheorem{example}{Example}

\newtheorem{corollary}{Corollary}

\begin{document}
\author{ Ushangi Goginava}
\address{U. Goginava, Department of Mathematical Sciences \\
United Arab Emirates University, P.O. Box No. 15551\\
Al Ain, Abu Dhabi, UAE}
\email{zazagoginava@gmail.com; ugoginava@uaeu.ac.ae}
\author{ Farrukh Mukhamedov}
\address{F. Mukhamedov, 
New Uzbekistan University, 54, Mustaqillik ave. \\
 100007, Tashkent, Uzbekistan;
Central Asian University, Tashkent 111221, Uzbekistan;
United Arab Emirates University, P.O. Box No. 15551\\
Al Ain, Abu Dhabi, UAE}
\email{farrukh.m@uaeu.ac.ae}

\title[limits of sequences of operators]{limits of sequences of operators
associated with Walsh System}
\date{}

\begin{abstract}
The  aim of the current paper is to determine the necessary and sufficient
conditions for the weights $\mathbf{q}=\{q_k\}$, ensuring that the sequence of
operators $\left\{ T_{n}^{\left( \mathbf{q}\right) }f\right\} $ associated with Walsh system, is
convergent almost everywhere for all integrable function $f$. The article also examines the convergence of a sequence of tensor product operators denoted as $\left\{ T_{n}^{( \mathbf{q})}\otimes T_{n}^{( \mathbf{p})}\right\}$ involving functions of two variables.
We point out that recent research by G\'{a}t and Karagulyan (2016) demonstrated that this sequence of tensor product operators cannot converge almost everywhere for every integrable function. In this paper, the necessary and sufficient conditions for the weight are provided which ensure that the sequence of the mentioned operators converges in measure on $L_{1}$.
\end{abstract}

\maketitle

\footnotetext{%
}

\section{\protect\bigskip Introduction}

It is well-known that Kolmogorov \cite{kolmogoroff1923serie} solved Lusin's
problem related to the trigonometric Fourier series. Namely, he was able
construct an integrable function whose trigonometric Fourier series diverges
almost everywhere (a.e.). Later on, he modified his proof to show the
existence of an integrable function for which its trigonometric Fourier
diverges at every point \cite{kolmogororf1926s}. In parallel to
trigonometric system, the Walsh system has advantages in the approximation
theory. Therefore, the convergence of Walsh-Fourier series has gained
serious attention amongst many mathematicians. For instance, Stein \cite%
{stein1961limits} was able to establish an analogue of the mentioned
Kolmogorov's result for the Walsh-Fourier series, i.e. it was proved the
existence of an integrable function whose Walsh-Fourier series is divergent
at all points. Let us mention the important results concerning trigonometric
and Walsh systems: Konyagin \cite{konyagin2000everywhere,konyagin2006almost}%
, Bochkarov \cite{bochkarev1975logarithmic,bochkarev2004everywhere}, K{ö}%
rner \cite{korner1981everywhere}. The issues of divergence of Fourier series
on subsequences were studied by Gosselin \cite{gosselin1958divergence},
Totik \cite{totik1982divergence}, Konyagin \cite{kon} Lie \cite%
{lie2012pointwise}, Di Plinio \cite{di2014lacunary}, Goginava, Oniani \cite%
{goginava2021divergence}, Oniani \cite{onian}.

The problem of divergence of Fourier series with respect to a general
orthonormal system was raised by Alexits in his monograph \cite%
{alexits1960konvergenzprobleme}. The problem posed by him is related to the
behavior of the Lebesgue constant of the orthonormal system $\Phi =\left\{
\varphi _{n}\left( x\right) :n\geq 1\right\} $, i.e.%
\begin{equation*}
L_{n}\left( x,\Phi \right) =\int\limits_{0}^{1}\left\vert
\sum\limits_{j=1}^{n}\varphi _{j}\left( x\right) \varphi _{j}\left( t\right)
\right\vert dt\text{ \ \ for }x\in \left[ 0,1\right] .
\end{equation*}

Alexits problem is formulated as follows: Let $\Phi =\left\{ \varphi
_{n}\left( x\right) :n\geq 1\right\} $ be any orthonormal system for which%
\begin{equation}
\overline{\lim\limits_{n\rightarrow \infty }}L_{n}\left( x,\Phi \right)
=+\infty  \label{l}
\end{equation}%
everywhere. Is there an integrable function whose Fourier series is
divergent almost everywhere?

In 1967, Olevskii \cite{olevskii1967order} proved that for any uniformly
bounded orthonormal system there exists a set $E\subset \left[ 0,1\right]
,\mu \left( E\right) >0$ such that condition (\ref{l}) holds for all $x\in E$.

The issue of the divergence of Fourier series with respect to a general
orthonormal system was studied by Bochkarev \cite{bochkarev1975fourier}.
Namely, he considered a uniformly bounded orthonormal system $\Phi =\left\{
\varphi _{n}\left( x\right) :n\geq 1\right\} $, the it is established the
existence of an integrable function $f$ whose Fourier series with respect to
the system $\Phi $ diverges on some $E\subset \left[ 0,1\right]$ with $%
\left\vert E\right\vert >0$. Later on, Kazarian \cite{kazarian1982some}
showed that, in Bochkarov's theorem, the divergence for a positive measure
cannot be replaced by a divergence almost everywhere, which finally gave a
negative answer to the Alexits's problem.

Nonetheless, when examining weighted averages of orthonormal systems, a
compelling question arises: can the Alexits's problem be positively resolved
for these weighted averages? 
In the current paper, we are going to solve such a kind of problem within the Walsh system. 
To be more precise, let us introduce necessary notations. 

Let $f\in L_{1}\left( \mathbb{I}\right) $, $%
\mathbb{I}:=[0,1)$ and $S_{k}(f;x)$ denotes the $k$th partial sums of the
Fourier series with respect to the Walsh system. Assume that $\mathbf{q}%
:=\{q_{k}:k\geq 0\}$ be a sequence of non-negative numbers. Let us consider
a sequence of linear operators defined by
\begin{equation}
T_{n}^{\left( \mathbf{q}\right) }(f;x):=\frac{1}{Q_{n}}%
\sum_{k=1}^{n}q_{n-k}S_{k}(f;x),\quad n\in \mathbb{N},  \label{seq}
\end{equation}%
where
\begin{equation*}
Q_{n}:=\sum_{k=0}^{n-1}q_{k}.
\end{equation*}%
Throughout the article, the sequence $\mathbf{q}$ enjoys the following
assumption:

\begin{itemize}
\item[(A)] the limit $\lim\limits_{k\rightarrow \infty }\left(
Q_{2^{k}}/Q_{2^{k-1}}\right) $ exists.
\end{itemize}

Now let us introduce the notion of \textit{the Lebesgue sequence} for the
sequence (\ref{seq}):%
\begin{equation}  \label{Lq}
L_{n}^{\left( \mathbf{q}\right) }:=\int\limits_{\mathbb{I}}\left\vert \frac{1%
}{Q_{n}}\sum_{k=1}^{n}q_{n-k}D_{k}(x)\right\vert dx.
\end{equation}

\begin{remark}
It's important to emphasize that if $\mathbf{q}$ is non-decreasing, then $%
T_{n}^{\left( \mathbf{q}\right) }(f;x)$ is convergent almost everywhere for
all $f\in L_{1}$ \cite[p. 67]{hardy2000divergent}. Hence, it becomes
pertinent to explore the scenario where $\mathbf{q}$ is non-increasing.
\end{remark}

Let us consider two important examples.

Assume that
\begin{equation}
q_{j}=\left\{
\begin{array}{c}
1,\ \ j=0 \\
0, \ \ j>0%
\end{array}%
\right. .  \label{ps}
\end{equation}

Then $T_{n}^{\left( \mathbf{q}\right) }(f;x)$ coincides with the partial
sums $S_{n}(f;x)$. Its divergence was established by Stein \cite%
{stein1961limits}. Moreover, its Lebesgus sequence diverges as well, i.e. $%
\sup\limits_{n\in \mathbb{N}}L_{n}^{\left( \mathbf{q}\right) }=\infty $.%
\newline

Now, we consider another decreasing sequence $q_{k}=1/k$. Then
\begin{equation}
T_{n}^{\left( \mathbf{q}\right) }(f;x)=\frac{1}{\log n}\sum%
\limits_{k=1}^{n-1}\frac{S_{k}\left( f,x\right) }{n-k}.  \label{lm}
\end{equation}%
In \cite{GatGogiAMSDiv2} it has been furnished an integrable function $f$
for which%
\begin{equation*}
\sup\limits_{n}\left\vert T_{n}^{\left( \mathbf{q}\right) }(f;x)\right\vert
=\infty \text{ a. e. on }\mathbb{I}\text{.}
\end{equation*}%
Moreover, for the corresponding Lebesgue sequence one has $\sup\limits_{n\in
\mathbb{N}}L_{n}^{\left( \mathbf{q}\right) }=\infty $.\newline

The provided examples naturally clarifies the mentioned problem for the
sequence (\ref{seq}). Namely, \textit{let us suppose that the condition
\begin{equation}
\sup\limits_{n\in \mathbb{N}}L_{n}^{\left( \mathbf{q}\right) }=\infty
\label{lf}
\end{equation}%
is satisfied for (\ref{seq}). Is it possible to construct an integrable
function $f$ for which the sequence $\{T_{n}^{\left( \mathbf{q}\right) }(f)\}
$ is divergent almost everywhere?}

In this paper, we positively solve the mentioned problem.

\begin{theorem}
\label{T1} Assume that $\mathbf{q}=\{q_{k}:k\geq 0\}$ is a non-increasing
sequence with (A). Let us suppose that the condition (\ref{lf}) is
satisfied. Then there exists a function $f\in L_{1}\left( \mathbb{I}\right) $
such that the sequence $\{T_{n}^{\left( \mathbf{q}\right) }(f)\}$ (see (\ref%
{seq})) diverges almost everywhere.
\end{theorem}

It turns out that the condition (\ref{lf}) can be characterized through $%
\mathbf{q}=\{q_{k}:k\geq 0\}$. Namely, we have the next result.

\begin{lemma}
\label{QQ} Assume that $\mathbf{q}=\{q_{k}:k\geq 0\}$ is a non-increasing
sequence with (A). Then, the following statements are equivalent:

\begin{itemize}
\item[(i)] $\sup\limits_{n\in \mathbb{N}}L_{n}^{\left( \mathbf{q}\right)
}=\infty ;$

\item[(ii)] $\rho \left( \mathbf{q}\right) :=\sup\limits_{n\in \mathbb{N}%
}\left( \frac{1}{Q_{2^{n}}}\sum\limits_{k=0}^{n}Q_{2^{k}}\right) =\infty ;$

\item[(iii)] $\lim\limits_{k\rightarrow \infty }\frac{Q_{2^{k}}}{Q_{2^{k-1}}}%
=1;$

\item[(iv)] $\lim\limits_{k\rightarrow \infty }\frac{2^{k}q_{2^{k}}}{%
Q_{2^{k}}}=0.$
\end{itemize}
\end{lemma}

\begin{proof}
The equivalence of (i) and (ii) is proved in \cite{goginava2023some}. Let us
establish (ii)$\Rightarrow$ (iii). Assume that (iii) does not hold, then one
can find $q^{\prime }$ $\in \left( 0,1\right) $ such that
\begin{equation*}
Q_{2^{l}}\leq \left( q^{\prime }\right) ^{k-l}Q_{2^{k}}\left( k_{0}\leq
l\leq k\right) ,
\end{equation*}%
which yields $\rho \left( \mathbf{q}\right) <\infty $, and consequently the
condition (ii) also does not hold.

Due to
\begin{equation*}
\frac{2^{k-1}q_{2^{k}}}{Q_{2^{k-1}}}\leq \frac{Q_{2^{k}}-Q_{2^{k-1}}}{%
Q_{2^{k-1}}}\leq \frac{2^{k-1}q_{2^{k-1}}}{Q_{2^{k-1}}}.
\end{equation*}%
the equivalence of (iii) and (iv) is immediately obtained.

Finally, let us prove that (iv)$\Rightarrow $ (ii). Denote $\delta
_{k}:=2^{k}q_{2^{k}}/Q_{2^{k}}$. Then
\begin{equation}
\frac{1}{Q_{2^{n}}}\sum\limits_{k=0}^{n}Q_{2^{k}}=\frac{1}{Q_{2^{n}}}%
\sum\limits_{k=0}^{n}\frac{2^{k}q_{2^{k}}}{\delta _{k}}.  \label{rho1}
\end{equation}%
Using the monotonicity of $\mathbf{q}$ one can find
\begin{equation*}
c_{1}\leq \frac{1}{Q_{2^{n}}}\sum\limits_{k=0}^{n}2^{k}q_{2^{k}}\leq
c_{2}\left( n\geq n_{0}\right) ,
\end{equation*}%
where $c_{1},c_{2}>0$. Now, the last one together with $1/\delta
_{k}\rightarrow \infty \left( k\rightarrow \infty \right) $ and \eqref{rho1}
yield $\rho \left( \mathbf{q}\right) =\infty $.

This completes the proof.
\end{proof}

We stress that if the condition (\ref{lf}) is not fulfilled, then for every
integrable function, the operator sequence (\ref{seq}) is convergent almost
everywhere \cite{goginava2023some}.

Now, taking into account this result together with Lemma \ref{QQ}, we may
formulated the next result about the convergence.

\begin{corollary}
Assume that $\mathbf{q}=\{q_{k}:k\geq 0\}$ is a non-increasing sequence with
(A). Then the following three statements are equivalent:

\begin{itemize}
\item[(i)] $\sup\limits_{n\in \mathbb{N}}L_{n}^{\left( \mathbf{q}\right)
}<\infty ;$\newline

\item[(ii)] $\rho \left( \mathbf{q}\right) <\infty ;$

\item[(iii)] For every $f\in L_1(\mathbb{I})$, the operator sequence $%
\{T_{n}^{\left( \mathbf{q}\right) }(f)\}$ given by (\ref{seq}) converges
almost everywhere.
\end{itemize}
\end{corollary}

\begin{remark}
It is well-known that the convergence in measure is weaker that a.e.
convergence, so it is reasonable to explore the convergence of trigonometric
Fourier series in measure. Again, this problem was successfully investigated
by Kolmogoroff \cite{kolmogoroff1925fonctions}. He proved that any $f\in
L_{1}$ its trigonometric Fourier series always converges in measure. The
similar kind of result has been established by Burkholder \cite{Burk} (see
also \cite[Ch. 3]{SWS}) for the Walsh systems in $L_{1}$. We point out that
in \cite{goginava2023some} it has been established that for any non-increasing sequence
$\mathbf{q}=\{q_{k}:k\geq 0\}$ the sequence $\{T_{n}^{\left( \mathbf{q}%
\right) }(f)\}$ given by (\ref{seq}) converges in measure for any  $f\in L_1(%
\mathbb{I})$.
\end{remark}

However, in the two-dimensional setting, the situation drastically changes.
Konyagin \cite{konyagin1993subsequence} and Getsadze \cite%
{getsadze1992divergence} (independently of each other) have proved the
existence of two-variable integrable function whose square partial sums of
trigonometric Fourier (Walsh-Fourier) series diverges in measure. The
considered square partial sums of the two-dimensional trigonometric (Walsh)
Fourier series can be rewritten as tensor product of one-dimensional
trigonometric Fourier series. This fact was a key point in the mentioned
papers. Such kind of representation led the consideration of sequences of
operators acting on double-variable $L_{1}$-spaces which are represented as
tensor product of one-variable operators. It turned out that such kind of
sequences may converge or diverge in measure on $L_{1}$. Recently, G\'at and
Kargulyan \cite{gat2016convergence}, by employing tensor product technique,
have proved a.e. divergence of matrix transforms of rectangular partial sums
of Fourier series with respect to complete orthonormal system.


The second focus of this paper explores sequences of linear operators
defined as partial sums linked to the weighted Walsh-Fourier series %
\eqref{seq}. Examples of such sequences include two-dimensional Fej'er
means, Cesáro means, logarithmic means, Nörlund means, and others widely
recognized in the field.


The second main aim of the current paper is to determine the necessary and
sufficient conditions for the weights, ensuring that the discussed sequence
of operators converges in measure on $L_{1}$.

To formulate the result, let us recall some notions.

The Kronecker product $\left( w_{n,m};n,m\in \mathbb{N}\right) $ of two
Walsh systems is said to be the two-dimensional Walsh system. Thus
\begin{equation*}
w_{n,m}(x,y):=w_{n}(x)\otimes w_{m}(y).
\end{equation*}%
Define
\begin{equation}
\left( S_{n}\right) _{1}f\left( x,y\right) :=S_{n}f\left( \cdot ,y\right) .
\label{1}
\end{equation}%
In this right side of (\ref{1}) $f\left( \cdot ,y\right) $ is considered as
a function in the variable $x$ and the other variable is fixed almost
everywhere. Analogously, one can define
\begin{equation*}
\left( S_{n}\right) _{2}f\left( x,y\right) :=S_{n}f\left( x,\cdot \right) .
\end{equation*}

Given two monotone non-increasing sequences $\mathbf{q}=\left\{ q_{n}:n\in
\mathbb{N}\right\} $, $\mathbf{p}=\left\{ p_{n}:n\in \mathbb{N}\right\} $
let us set
\begin{equation}
\left( T_{n}^{\left( \mathbf{q}\right) }\right) _{1}:=\frac{1}{Q_{n}}%
\sum_{k=1}^{n}q_{n-k}\left( S_{k}\right) _{1},\ \ \left( T_{n}^{\left(
\mathbf{p}\right) }\right) _{2}:=\frac{1}{P_{n}}\sum_{k=1}^{n}p_{n-k}\left(
S_{k}\right) _{2}\quad (n\in {\mathbb{N}}).  \label{TT}
\end{equation}%
Their tensor product is defined by%
\begin{equation*}
T_{n,m}^{\left( \mathbf{q,p}\right) }:=\left( T_{n}^{\left( \mathbf{q}%
\right) }\right) _{1}\circ \left( T_{m}^{\left( \mathbf{p}\right) }\right)
_{2}.
\end{equation*}%
It is easy to see that%
\begin{equation*}
T_{n,m}^{\left( \mathbf{q},\mathbf{p}\right) }\left( f\right) =f\ast \left( {%
F}_{n}^{\left( \mathbf{q}\right) }\otimes {F}_{m}^{\left( \mathbf{p}\right)
}\right) .
\end{equation*}

We denote by $L_{0}=L_{0}(\mathbb{I}^{2})$ the Lebesgue space of functions
that are measurable and finite almost everywhere on $\mathbb{I}^{2}$.
Moreover, $\boldsymbol{\mu }(A)$ stands for the two dimensional Lebesgue
measure of a set $A\subset \mathbb{I}^{2}$.

Denote
\begin{equation}
\pi _{k}=\pi _{k}\left( \mathbf{q}\right) :=\frac{Q_{2^{k}}-Q_{2^{k-1}}}{%
Q_{2^{k-1}}}  \label{eps}
\end{equation}%
and%
\begin{equation*}
\beta \left( \mathbf{q}\right) :=\lim\limits_{k\rightarrow \infty }\pi
_{k}\left( \mathbf{q}\right) .
\end{equation*}

\begin{theorem}
\label{convmeas} Assume that $\mathbf{p}=\left\{ p_{n}:n\in \mathbb{N}%
\right\} $ and $\mathbf{q}=\left\{ q_{n}:n\in \mathbb{N}\right\} $ are
monotone non-increasing sequences of non-negative numbers with (A). Then the
following statements hold:

\begin{enumerate}
\item[(i)] if
\begin{equation}
\beta \left( \mathbf{q}\right) +\beta \left( \mathbf{p}\right) >0.
\label{condmeasconv}
\end{equation}%
Then, for every Let $f\in L_{1}\left( \mathbb{I}^{2}\right) $, one has
\begin{equation*}
\left\vert T_{n,m}^{\left( \mathbf{q},\mathbf{p}\right) }\left( f,x,y\right)
-f\left( x,y\right) \right\vert \rightarrow 0\text{ in measure on }\mathbb{I}%
^{2},\text{ as }n,m\rightarrow \infty ;
\end{equation*}

\item[(ii)] if
\begin{equation}
\beta \left( \mathbf{q}\right) +\beta \left( \mathbf{q}\right) =0,
\label{divmeasure}
\end{equation}%
then the set of the functions from $L_{1}\left( \mathbb{I}^{2}\right) $ for
which the sequence $T_{n,m}^{\left( \mathbf{q},\mathbf{p}\right) }\left(
f\right) $ is convergent in measure on $\mathbb{I}^{2}$ forms the first
Baire category in $L_{1}\left( \mathbb{I}^{2}\right) $.
\end{enumerate}
\end{theorem}

\section{Proof of Theorem \protect\ref{T1}}

In this section, we are going to prove Theorem \ref{T1}. Before, to proceed
to do it, we first provide necessary definitions and auxiliary facts.

By $\mathbb{N}$ the set of non-negative integers is denoted. Given $k\in
\mathbb{N}$ and $x\in \mathbb{I},$ let $I_{k}(x)$ denote the dyadic interval
of length $2^{-k}$ which contains the point $x$. Also, we use the notation $%
I_{n}:=I_{n}\left( 0\right) \left( n\in \mathbb{N}\right) ,\overline{I}%
_{k}\left( x\right) :=\mathbb{I}\backslash I_{k}\left( x\right) $. Let
\begin{equation*}
x=\sum\limits_{n=0}^{\infty }x_{n}2^{-\left( n+1\right) }:=\left(
x_{0},x_{1},\dots\right)
\end{equation*}
be the dyadic expansion of $x\in \mathbb{I}$, where $x_{n}=0$ or $1$ and if $%
x$ is a dyadic rational number we choose the expansion which terminate in $%
0^{\prime }$s. In the sequel, we also frequently employ the notation $%
I_{k}(x)=I_{k}\left( x_{0},x_{1},\dots,x_{k-1}\right)$.

Given any $n\in \mathbb{N}$, one can write $n$ uniquely as
\begin{equation*}
n=\sum\limits_{k=0}^{\infty }\varepsilon _{k}\left( n\right) 2^{k}:=\left(
\varepsilon _{0}\left( n\right) ,\dots ,\varepsilon _{k}\left( n\right)
,\dots \right),
\end{equation*}
where $\varepsilon _{k}\left( n\right)\in\{0,1\}$, for $k\in \mathbb{N}$.
This expression is called \textit{the binary expansion} of $n$ and the
numbers $\varepsilon _{k}\left( n\right) $ are referred \textit{the binary
coefficients} of $n$. Furthermore, one denotes $\left\vert n\right\vert
:=\max \{j\in \mathbb{N}: \ \varepsilon _{j}\left( n\right) \neq 0\}$, that
is $2^{\left\vert n\right\vert }\leq n<2^{\left\vert n\right\vert +1}.$

We recall that the $n$th $\left( n\in \mathbb{N}\right) $ Walsh-Paley
function at point $x\in \mathbb{I}$ is defined by
\begin{equation*}
w_{n}\left( x\right) =\left( -1\right) ^{\sum\limits_{j=0}^{\infty
}\varepsilon _{j}\left( n\right) x_{j}}.
\end{equation*}%
By $\dotplus $ we denote the logical addition on $\mathbb{I}$, i.e.
\begin{equation*}
x\dotplus y:=\sum\limits_{n=0}^{\infty }\left\vert x_{n}-y_{n}\right\vert
2^{-\left( n+1\right) }
\end{equation*}
where $x,y\in \mathbb{I}$.

The Walsh-Dirichlet kernel is defined by $D_{n}:=\sum%
\limits_{k=0}^{n-1}w_{k}.$ It is well-known that \cite[p. 7]{SWS}
\begin{equation}
D_{2^{n}}\left( x\right) =2^{n}\chi _{I_{n}}\left( x\right) ,  \label{Dir}
\end{equation}%
where $\chi _{E}$ is the indicator function of a set $E$. Set $D_{n}^{\ast
}:=w_{n}D_{n}$.

In this section, the one dimensional Lebesgue measure of a set $A\subset
\mathbb{I}$ will be denoted by $\mu \left( A\right) $.

\begin{proof}[Proof of Theorem \protect\ref{T1}]
For every $l<k$, we define%
\begin{equation}
\delta \left( l,k\right) :=\max \left\{ \pi _{s}:k-l\leq s\leq k\right\} ,
\label{delta}
\end{equation}%
where $\pi _{s}$ is given by \eqref{eps}.

We observe that Lemma \ref{QQ} implies the sequence $\{\pi _{k}\}$ tends to
zero; however, this convergence might not be monotonic.

The new sequence has the following property
\begin{equation*}
\delta \left( l+1,k\right) \geq \delta \left( l,k\right) ,\quad l=0,\dots
,k-1.
\end{equation*}%
Now, let us define
\begin{equation*}
\gamma _{k}:=\min \left\{ 0\leq l\leq k:\delta \left( l+1,k\right) >\frac{1}{%
l+1}\right\} .
\end{equation*}%
We know that the sequence $\left\{ \frac{1}{l+1}:0\leq l<k\right\} $ is
decreasing, while $\left\{ \delta \left( l,k\right) :\ 0\leq l<k\right\} $
is increasing, therefore, if they "meet", then $\gamma _{k}$ would be such a
"meeting" point. In that case, one has
\begin{equation*}
\gamma _{k}\uparrow \infty \text{ as }k\rightarrow \infty ,
\end{equation*}%
\begin{equation*}
\delta \left( \gamma _{k}+1,k\right) >\frac{1}{1+\gamma _{k}}
\end{equation*}%
and%
\begin{equation}
\delta \left( \gamma _{k},k\right) \leq \frac{1}{\gamma _{k}}.  \label{meet}
\end{equation}%
If those two sequences do not "meet" with each other, then we let $\gamma
_{k}:=\left[ k/2\right] $. Hence,
\begin{equation*}
\delta \left( \gamma _{k},k\right) <\frac{1}{\gamma _{k}}.
\end{equation*}%
Now, we can write%
\begin{eqnarray*}
Q_{2^{k}} &=&\left( 1+\pi _{k}\right) Q_{2^{k-1}}=\cdots  \\
&=&\left( 1+\pi _{k}\right) \cdots \left( 1+\pi _{k-\gamma _{k}+1}\right)
Q_{2^{k-\gamma _{k}}}.
\end{eqnarray*}%
From (\ref{delta}) and (\ref{meet}) one finds
\begin{eqnarray*}
\ln \left[ \left( 1+\pi _{k}\right) \cdots \left( 1+\pi _{k-\gamma
_{k}+1}\right) \right]  &\sim &\pi _{k}+\cdots +\pi _{k-\gamma _{k}+1} \\
&\leq &c\delta \left( \gamma _{k},k\right) \gamma _{k}\leq c<\infty .
\end{eqnarray*}%
Consequently,%
\begin{equation}
Q_{2^{k-s}}\geq Q_{2^{k-\gamma _{k}}}\geq cQ_{2^{k}},\ \ s=1,\dots ,\gamma
_{k}.  \label{l1}
\end{equation}%
Without lost the generality, we may assume that $\gamma _{k}$ is an even
number, and for the sake of simplicity of the notation, we will use $2\gamma
_{k}$ instead of $\gamma _{k}$. Hence,%
\begin{equation}
Q_{2^{k-s}}\geq cQ_{2^{k}},\ \ s=1,\dots ,2\gamma _{k}.  \label{l2}
\end{equation}%
Set%
\begin{equation*}
\theta _{i}:=\left( u_{0},\dots ,u_{N-2\gamma _{N}-1},\underbrace{%
u_{N-2\gamma _{N}},\dots ,u_{N-\gamma _{N}-1}},\underbrace{u_{N-2\gamma
_{N}},\dots ,u_{N-\gamma _{N}-1}},0,\dots \right) ,
\end{equation*}%
where%
\begin{equation*}
i:=\sum\limits_{l=0}^{N-\gamma _{N}-1}u_{l}2^{N-\gamma _{N}-l-1}.
\end{equation*}%
Now, let us consider the Walsh Polynomial $P_{N}\left( t\right) $ defined by%
\begin{equation*}
P_{N}\left( t\right) :=\frac{1}{2^{N-\gamma _{N}}\sqrt{\gamma _{N}}}%
\sum\limits_{i=0}^{N-\gamma _{N}-1}D_{2^{N}}\left( t\dotplus \theta
_{i}\right) \sum\limits_{l=N-2\gamma _{N}}^{N-\gamma _{N}-1}w_{2^{l}}\left(
t\right) .
\end{equation*}%
From the orthogonality of Walsh functions, one finds
\begin{eqnarray}
\left\Vert P_{N}\right\Vert _{1} &\leq &\frac{2^{N}}{2^{N-\gamma _{N}}\sqrt{%
\gamma _{N}}}\sum\limits_{i=0}^{N-\gamma _{N}-1}\left\Vert \chi
_{I_{N}\dotplus \theta _{i}}\sum\limits_{l=N-2\gamma _{N}}^{N-\gamma
_{N}-1}w_{2^{l}}\right\Vert _{1}  \label{L1} \\
&=&\frac{1}{\sqrt{\gamma _{N}}}\sum\limits_{i=0}^{N-\gamma _{N}-1}\left\Vert
\chi _{I_{N}-\gamma _{N}\dotplus \theta _{i}}\sum\limits_{l=N-2\gamma
_{N}}^{N-\gamma _{N}-1}w_{2^{l}}\right\Vert _{1}  \notag \\
&=&\frac{1}{\sqrt{\gamma _{N}}}\left\Vert \sum\limits_{l=N-2\gamma
_{N}}^{N-\gamma _{N}-1}w_{2^{l}}\right\Vert _{1}  \notag \\
&\leq &\frac{1}{\sqrt{\gamma _{N}}}\left\Vert \sum\limits_{l=N-2\gamma
_{N}}^{N-\gamma _{N}-1}w_{2^{l}}\right\Vert _{2}=1.  \notag
\end{eqnarray}%
On the other hand,
\begin{equation}
\left\Vert P_{N}\right\Vert _{\infty }\leq c2^{\gamma _{N}}\sqrt{\gamma _{N}}%
.  \label{est22}
\end{equation}

Denote%
\begin{equation*}
E^{\prime }:=\overline{\lim\limits_{N\rightarrow \infty }}E_{N},
\end{equation*}%
where
\begin{equation*}
E_{N}:=\bigcup\limits_{a_{0}=0}^{1}\cdots \bigcup\limits_{a_{_{N-\gamma
_{N}-1}}=0}^{1}I_{_{N-\gamma _{N}+1}}\left( a_{0},...,a_{_{N-\gamma
_{N}-1}},1-a_{_{N-2\gamma _{N}}}\right) ,
\end{equation*}%
We can observe that
\begin{equation*}
\mu \left( E^{\prime }\right) =1.
\end{equation*}

Let $x\in E^{\prime }$. Then, we choose a strictly increasing sequence $%
\left\{ N_{k}:k\geq 1\right\} $ which enjoys the following conditions:
\begin{equation}
N_{k}-2\gamma _{N_{k}}>N_{k-1}  \label{Nk>}
\end{equation}%
\begin{equation*}
x\in E_{N_{k}},\ k=1,2,\dots
\end{equation*}%
\begin{equation}
\gamma _{N_{k}}\geq k^{8},  \label{2}
\end{equation}%
\begin{equation}
\sum\limits_{j=1}^{k-1}\sqrt[4]{\gamma _{N_{j}}}2^{\gamma _{N_{j}}}<\frac{%
\sqrt[4]{\gamma _{N_{k}}}}{k}.  \label{3}
\end{equation}%
In the subsequent discussion, to simplify matters, we employ $\gamma _{k}$
instead of $\gamma _{N_{k}}$.

Denote
\begin{eqnarray}
n\left( N_{k},x\right) &:&=\underbrace{\varepsilon _{N_{k}-2\gamma
_{k}}\left( x\right) 2^{N_{k}-2\gamma _{k}}+\cdots +\varepsilon
_{N_{k}-\gamma _{k}-1}\left( x\right) 2^{N_{k}-\gamma _{k}-1}}%
_{\lambda_{N_{k}}}  \label{nu} \\
&&+\varepsilon _{N_{k}-2\gamma _{k}}\left( x\right) 2^{N_{k}-\gamma
_{k}}+\cdots +\varepsilon _{N_{k}-\gamma _{k}-1}\left( x\right) 2^{N_{k}-1}
\notag \\
&:&=\lambda _{N_{k}}+2^{\gamma _{k}}\lambda _{N_{k}},  \notag
\end{eqnarray}%
where $\varepsilon _{k}\left( x\right) $ will be determined later.

It is easily seen that the operators $T_{n}^{\left( \mathbf{q}\right) }(f)$
can be represented as follows:
\begin{equation}  \label{O}
T_{n}^{\left( \mathbf{q}\right) }(f,x)=\left( f\ast F_{n}^{\left( \mathbf{q}%
\right) }\right) \left( x\right) ,
\end{equation}%
where%
\begin{equation*}
{F}_{n}^{\left( \mathbf{q}\right) }(t):=\frac{1}{Q_{n}}%
\sum_{k=1}^{n}q_{n-k}D_{k}(t).
\end{equation*}

Following \cite{goginava2023some} we may write
\begin{equation}
{F}_{n}^{\left( \mathbf{q}\right) }=F_{n,1}^{\left( \mathbf{q}\right)
}+F_{n,2}^{\left( \mathbf{q}\right) },\ \ \ n\in \mathbb{N},  \label{O1}
\end{equation}%
where
\begin{eqnarray}
&&F_{n,1}^{\left( \mathbf{q}\right) }:=\frac{w_{n}}{Q_{n}}\sum_{j=1}^{\infty
}\varepsilon _{j}\left( n\right) Q_{n\left( j\right) }w_{2^{j}}D_{2^{j}},
\label{O2} \\
&&F_{n,2}^{\left( \mathbf{q}\right) }:=-\frac{w_{n}}{Q_{n}}%
\sum_{j=1}^{r}\varepsilon _{j}\left( n\right) w_{n\left( j\right)
}w_{2^{j}-1}\sum_{k=1}^{2^{j}-1}q_{k+n\left( j-1\right) }D_{k}  \notag
\end{eqnarray}%
and
\begin{equation*}
n\left( s\right) :=\sum\limits_{j=0}^{s}\varepsilon _{j}\left( n\right)
2^{j},\ \ s\in \mathbb{N}.
\end{equation*}

The representation (\ref{O1}) together with (\ref{O}) allows the operator $%
T_{n}^{\left( \mathbf{q}\right) }(f)$ to write a sum of two operators $f\ast
F_{n,1}^{\left( \mathbf{q}\right) }$ and $f\ast F_{n,2}^{\left( \mathbf{q}%
\right) }$. We notice that $f\ast F_{n,1}^{\left( \mathbf{q}\right) }$ is a
martingale transformation, and the second one $f\ast F_{n,2}^{\left( \mathbf{%
q}\right) }$ has the "good" property that the mapping $f\in L_{1}(\mathbb{I}%
)\rightarrow \sup\limits_{n\in \mathbb{N}}\left\vert f\ast F_{n,2}^{\left(
\mathbf{q}\right) }\right\vert $ has of weak type $\left( L_{1},L_{1}\right)
$ (see \cite{goginava2023some}). Namely, there exists an absolute constant $c>0
$ such that%
\begin{equation*}
\left\Vert f\ast F_{n,2}^{\left( \mathbf{q}\right) }\right\Vert _{1,\infty
}:=\sup\limits_{t>0}t\mu \left\{ x\in \mathbb{I}:\sup\limits_{n\in \mathbb{N}%
}\left\vert f\ast F_{n,2}^{\left( \mathbf{q}\right) }\right\vert >t\right\}
\leq c\left\Vert f\right\Vert _{1}\text{ \ }\left( f\in L_{1}\right) \text{,}
\end{equation*}%
which implies
\begin{equation}
\sup\limits_{n}\left\vert f\ast {F}_{n,2}^{\left( \mathbf{q}\right)
}\right\vert <\infty   \label{ae}
\end{equation}%
a. e. for $f\in L_{1}\left( \mathbb{I}\right) .$

Therefore, to prove the theorem, it is enough to establish the divergence of
the sequence $f\ast {F}_{n,1}^{\left( \mathbf{q}\right) }$.

Put
\begin{equation*}
F_{n\left( N_{k},x\right) ,1}^{\left( \mathbf{q}\right) \ast }:=w_{n\left(
N_{k},x\right) }F_{n\left( N_{k},x\right) ,1}^{\left( \mathbf{q}\right) }.
\end{equation*}%
Due to
\begin{equation*}
\left( x\dotplus t\right) _{N_{k}-2\gamma _{k}}\dotplus \left( x\dotplus
t\right) _{N_{k}-\gamma _{k}}=1,\ \ t\in \mathop{supp}\left(
P_{N_{k}}\right) ,\ x\in E^{\prime },
\end{equation*}%
by (\ref{Dir}) and (\ref{nu}) we infer that
\begin{equation*}
F_{n\left( N_{k},x\right) ,1}^{\left( \mathbf{q}\right) \ast }\left(
x\dotplus t\right) =\frac{Q_{\lambda _{N_{k}}}}{Q_{n\left( N_{k},x\right) }}%
F_{\lambda _{N_{k}},1}^{\left( \mathbf{q}\right) \ast }\left( x\dotplus
t\right) .
\end{equation*}%
Consequently,{\small
\begin{eqnarray}
P_{N_{k}}\ast F_{n\left( N_{k},x\right) ,1}^{\left( \mathbf{q}\right) \ast }
&=&\frac{1}{2^{N_{k}-\gamma _{k}}\sqrt{\gamma _{k}}}\sum%
\limits_{i=0}^{N_{k}-\gamma _{k}-1}\sum\limits_{l=N_{k}-2\gamma
_{k}}^{N_{k}-\gamma _{k}-1}w_{2^{l}}\left( x\right)   \label{lower} \\
&&\times D_{2^{N_{k}}}\left( \cdot \dotplus \theta _{i}\right) \ast \left(
w_{2^{l}}F_{n\left( N_{k},x\right) ,1}^{\left( \mathbf{q}\right) \ast
}\right)   \notag \\
&=&\frac{1}{2^{N_{k}-\gamma _{k}}\sqrt{\gamma _{k}}}\sum%
\limits_{i=0}^{N_{k}-\gamma _{k}-1}\sum\limits_{l=N_{k}-2\gamma
_{k}}^{N_{k}-\gamma _{k}-1}w_{2^{l}}\left( x\right)   \notag \\
&&\times \left( w_{2^{l}}\left( \cdot \dotplus \theta _{i}\right) F_{n\left(
N_{k},x\right) ,1}^{\left( \mathbf{q}\right) \ast }\left( \cdot \dotplus
\theta _{i}\right) \right) \ast D_{2^{N_{k}}}  \notag \\
&=&\frac{1}{\sqrt{\gamma _{k}}}\sum\limits_{l=N_{k}-2\gamma
_{k}}^{N_{k}-\gamma _{k}-1}w_{2^{l}}\left( x\right)   \notag \\
&&\times \frac{1}{Q_{n\left( N_{k},x\right) }}\frac{1}{2^{N_{k}-\gamma _{k}}}%
\sum\limits_{i=0}^{N_{k}-\gamma _{k}-1}w_{2^{l}}\left( x\dotplus \theta
_{i}\right) Q_{\lambda _{N_{k}}}F_{\lambda _{N_{k}},1}^{\left( \mathbf{q}%
\right) \ast }\left( x\dotplus \theta _{i}\right) .  \notag
\end{eqnarray}%
}{\normalsize The equality }{\small
\begin{equation*}
w_{2^{l}}\left( x\dotplus \theta _{i}\right) Q_{\lambda _{N_{k}}}F_{\lambda
_{N_{k}},1}^{\left( \mathbf{q}\right) \ast }\left( x\dotplus \theta
_{i}\right) =2^{N_{k}-\gamma _{k}}\int\limits_{I_{N_{k}-\gamma _{k}}\left(
\theta _{i}\right) }w_{2^{l}}\left( x\dotplus t\right) Q_{\lambda
_{N_{k}}}F_{\lambda _{N_{k}},1}^{\left( \mathbf{q}\right) \ast }\left(
x\dotplus t\right) dt
\end{equation*}%
}{\normalsize and (\ref{O2}) together with orthogonality of the Walsh
functions imply
\begin{eqnarray*}
&&\sum\limits_{i=0}^{N_{k}-\gamma _{k}-1}w_{2^{l}}\left( x\dotplus \theta
_{i}\right) Q_{\lambda _{N_{k}}}F_{\lambda _{N_{k}},1}^{\left( \mathbf{q}%
\right) \ast }\left( x\dotplus \theta _{i}\right)  \\
&=&2^{N_{k}-\gamma _{k}}\sum\limits_{i=0}^{N_{k}-\gamma
_{k}-1}\int\limits_{I_{N_{k}-\gamma _{k}}\left( \theta _{i}\right)
}w_{2^{l}}\left( x\dotplus t\right) Q_{\lambda _{N_{k}}}F_{\lambda
_{N_{k}},1}^{\left( \mathbf{q}\right) \ast }\left( x\dotplus t\right) dt \\
&=&2^{N_{k}-\gamma _{k}}\int\limits_{\mathbb{I}}w_{2^{l}}\left( x\dotplus
t\right) Q_{\lambda _{N_{k}}}F_{\lambda _{N_{k}},1}^{\left( \mathbf{q}%
\right) \ast }\left( x\dotplus t\right) dt \\
&=&2^{N_{k}-\gamma _{k}}\int\limits_{\mathbb{I}}w_{2^{l}}\left( t\right)
Q_{\lambda _{N_{k}}}F_{\lambda _{N_{k}},1}^{\left( \mathbf{q}\right) \ast
}\left( t\right) dt \\
&=&2^{N_{k}-\gamma _{k}}\sum\limits_{j=N_{k}-2\gamma _{k}}^{N_{k}-\gamma
_{k}-1}\varepsilon _{j}\left( x\right) Q_{\lambda _{N_{k}\left( j\right)
}}\int\limits_{\mathbb{I}}w_{2^{l}}\left( t\right) w_{2^{j}}\left( t\right)
D_{2^{j}}\left( t\right) dt \\
&=&2^{N_{k}-\gamma _{k}}\varepsilon _{l}\left( x\right) Q_{\lambda
_{N_{k}\left( l\right) }}.
\end{eqnarray*}%
}

{\normalsize Consequently, from (\ref{lower}) it follows that
\begin{equation*}
P_{N_{k}}\ast F_{n\left( N_{k},x\right) ,1}^{\left( \mathbf{q}\right) \ast }=%
\frac{1}{Q_{n\left( N_{k},x\right) }\sqrt{\gamma _{k}}}\sum%
\limits_{l=N_{k}-2\gamma _{k}}^{N_{k}-\gamma _{k}-1}w_{2^{l}}\left( x\right)
\varepsilon _{l}\left( x\right) Q_{\lambda_{N_{k}\left( l\right) }}.
\end{equation*}
}

{\normalsize Let us first consider the case $x_{N_{k}-2\gamma _{k}}+\cdots
+x_{N_{k}-\gamma _{k}-1}<\frac{\gamma _{k}}{3}$, and define digits $%
\varepsilon _{k}\left( x\right) $ by $\varepsilon _{k}\left( x\right)
=1-x_{k}$. Then, by (\ref{l2}) one finds
\begin{eqnarray}
\left\vert P_{N_{k}}\ast F_{n\left( N_{k},x\right) ,1}^{\left( \mathbf{q}%
\right) \ast }\left( x\right) \right\vert &\geq &\frac{c}{Q_{2^{N_{k}}}\sqrt{%
\gamma _{k}}}\sum\limits_{l=N_{k}-2\gamma _{k}}^{N_{k}-\left[ \left(
5/3\right) \gamma _{k}\right] -1}Q_{2^{l}}  \label{ll} \\
&\geq &\frac{cQ_{2^{N_{k}-2\gamma _{k}}}\gamma _{k}}{Q_{2^{N_{k}}}\sqrt{%
\gamma _{k}}}  \notag \\
&\geq &c\sqrt{\gamma _{k}}.  \notag
\end{eqnarray}
}

{\normalsize Now, assume that $x_{N_{k}-2\gamma _{k}}+\cdots
+x_{N_{k}-\gamma _{k}-1}\geq \frac{\gamma _{k}}{3}$. The digits $\varepsilon
_{k}\left( x\right) $ are defined by $\varepsilon _{k}\left( x\right) =x_{k}.
$ Analogously, we establish
\begin{equation}
\left\vert P_{N_{k}}\ast F_{n\left( N_{k},x\right) ,1}^{\left( \mathbf{q}%
\right) \ast }\left( x\right) \right\vert \geq c\sqrt{\gamma _{k}}.
\label{b}
\end{equation}%
Combining (\ref{ll}) and (\ref{b}), we obtain
\begin{equation}
\left\vert P_{N_{k}}\ast F_{n\left( N_{k},x\right) ,1}^{\left( \mathbf{q}%
\right) \ast }\left( x\right) \right\vert \geq c\sqrt{\gamma _{k}}\ \ \text{
for a.e. }\ \ x\in \mathbb{I}\text{.}  \label{a+b}
\end{equation}%
Using the equalities (see (\ref{nu}))%
\begin{equation*}
w_{n\left( N_{k},x\right) }\left( t\right) =\left( -1\right) ^{\varepsilon
_{N_{k}-2\gamma _{k}}\left( t_{N_{k}-2\gamma _{k}}+t_{N_{k}-\gamma
_{k}}\right) +\cdots +\varepsilon _{N_{k}-2\gamma _{k}}\left(
t_{N_{k}-\gamma _{k}-1}+t_{N_{k}-1}\right) }
\end{equation*}%
and%
\begin{equation*}
t_{N_{k}-2\gamma _{k}+j}=t_{N_{k}-\gamma _{k}+j},\ \ j=0,\dots,\gamma
_{k}-1,t\in \text{supp}\left( P_{N_{k}}\right)
\end{equation*}%
we infer that $w_{n\left( N_{k},x\right) }\left( t\right) P_{N_{k}}\left(
t\right) =P_{N_{k}}\left( t\right) $. Hence, (\ref{b}) yields
\begin{equation}
\left\vert P_{N_{k}}\ast F_{n\left( N_{k},x\right) ,1}^{\left( \mathbf{q}%
\right) }\left( x\right) \right\vert \geq c\sqrt{\gamma _{k}}\ \ \text{ for
a.e. }\ \ x\in \mathbb{I}.  \label{nostar}
\end{equation}
}

{\normalsize Let us define
\begin{equation*}
f\left( t\right) :=\sum\limits_{j=1}^{\infty }\frac{P_{N_{j}}\left( t\right)
}{\sqrt[4]{\gamma _{j}}}.
\end{equation*}%
By (\ref{L1}) and (\ref{2}) we conclude that $f\in L_{1}(\mathbb{I})$. }

{\normalsize Let $j>k.$ From (\ref{Nk>}) and the orthogonality of the Walsh
functions it follows that}{\small
\begin{eqnarray}  \label{jk}
P_{N_{j}}\ast F^{\left( \mathbf{q}\right) }_{n\left( N_{k},x\right) ,1} &=&%
\frac{1}{2^{-\gamma _{j}}\sqrt{\gamma _{j}}}\sum\limits_{i=0}^{N_{j}-\gamma
_{j}-1}\int\limits_{I_{N_{j}}\left( x\dotplus \theta _{i}\right)
}\sum\limits_{l=N_{j}-2\gamma _{j}}^{N_{j}-\gamma _{j}-1}w_{2^{l}}\left(
x\dotplus t\right) F^{\left( \mathbf{q}\right) }_{n\left( N_{k},x\right)
,1}\left( t\right) dt \\[2mm]
&=&\frac{1}{\sqrt{\gamma _{j}}}\sum\limits_{i=0}^{N_{j}-\gamma
_{j}-1}\int\limits_{I_{N_{j}-\gamma _{j}}\left( x\dotplus \theta _{i}\right)
}\sum\limits_{l=N_{j}-2\gamma _{j}}^{N_{j}-\gamma _{j}-1}w_{2^{l}}\left(
x\dotplus t\right) F^{\left( \mathbf{q}\right) }_{n\left( N_{k},x\right)
,1}\left( t\right) dt  \notag \\
&=&\frac{1}{\sqrt{\gamma _{j}}}\int\limits_{\mathbb{I}}\sum%
\limits_{l=N_{j}-2\gamma _{j}}^{N_{j}-\gamma _{j}-1}w_{2^{l}}\left(
x\dotplus t\right) F^{\left( \mathbf{q}\right) }_{n\left( N_{k},x\right)
,1}\left( t\right) dt=0.  \notag
\end{eqnarray}
}

{\normalsize Now, we assume that $j<k$, then {\small
\begin{eqnarray*}
P_{N_{j}}\ast F^{\left( \mathbf{q}\right) }_{n\left( N_{k},x\right) ,1} &=&%
\underbrace{\frac{1}{Q_{n\left( N_{k},x\right) }}\sum\limits_{s=N_{k}-2%
\gamma _{k}}^{N_{k}-\gamma _{k}-1}\varepsilon _{s}\left( x\right) Q_{n\left(
N_{k},x\right) \left( s\right) }\left( P_{N_{j}}\ast w_{n\left(
N_{k},x\right) }D_{2^{s}}^{\ast }\right) }_{I} \\
&&+\underbrace{\frac{1}{Q_{n\left( N_{k},x\right) }}\sum\limits_{s=N_{k}-%
\gamma _{k}}^{N_{k}-1}\varepsilon _{s-\gamma _{k}}\left( x\right) Q_{\left(
n\left( N_{k},x\right) \right) \left( s\right) }\left( P_{N_{j}}\ast
w_{n\left( N_{k},x\right) }D_{2^{s}}^{\ast }\right) }_{II}.
\end{eqnarray*}%
} }

{\normalsize The binary expansion of $n\left( N_{k},x\right) $ contains
binary coefficients in two blocks, where the binary coefficients are equal
in the first and second blocks. If in the first block, and therefore, in the
second block all the coefficients are zero then $I=II=0$. }

{\normalsize Now, if there is at least one $s_{0}$ in the first block such
that $\varepsilon _{s_{0}}\left( x\right) =1$, then
\begin{equation*}
w_{2^{s_{0}}+2^{s_{0}+\gamma _{k}}}D_{2^{s_{0}}}^{\ast }=w_{2^{s_{0}+\gamma
_{k}}}D_{2^{s_{0}}}.
\end{equation*}%
Hence,
\begin{equation*}
P_{N_{j}}\ast w_{2^{s_{0}+\gamma _{k}}}D_{2^{s_{0}}}=0, \ \ \left(
j<k,s_{0}\in \left\{ N_{k}-2\gamma _{k},\dots,N_{k}-\gamma _{k}-1\right\}
\right)
\end{equation*}%
and%
\begin{equation}
I=0.  \label{I}
\end{equation}%
Let%
\begin{equation*}
\overline{s}:=\max \left\{ s:\varepsilon _{s}=1,N_{k}-2\gamma _{k}\leq
s<N_{k}-\gamma _{k}\right\} .
\end{equation*}%
Then%
\begin{eqnarray}
II &=&\frac{Q_{n\left( N_{k},x\right) \left( \overline{s}+\gamma _{k}\right)
}}{Q_{n\left( N_{k},x\right) }}\left( P_{N_{j}}\ast D_{2^{\overline{s}%
+\gamma _{k}}}\right)  \label{ii} \\
&=&\frac{Q_{n\left( N_{k},x\right) \left( \overline{s}+\gamma _{k}\right) }}{%
Q_{n\left( N_{k},x\right) }}P_{N_{j}}.  \notag
\end{eqnarray}
}

{\normalsize Combining (\ref{est22}),  (\ref{I}) and (\ref{ii}),
it then follows from (\ref{3}) that
\begin{equation}
\left\vert \sum\limits_{j=1}^{k-1}\frac{\left( P_{N_{j}}\ast F^{\left(
\mathbf{q}\right) }_{n\left( N_{k},x\right) ,1}\right) }{\sqrt[4]{\gamma _{j}%
}}\right\vert \leq \sum\limits_{j=1}^{k-1}\frac{\left\vert
P_{N_{j}}\right\vert }{\sqrt[4]{\gamma _{j}}}\leq \sum\limits_{j=1}^{k-1}%
\sqrt[4]{\gamma _{j}}2^{\gamma _{j}}<\frac{\sqrt[4]{\gamma _{k}}}{k}.
\label{<}
\end{equation}
}

{\normalsize From (\ref{nostar}), (\ref{jk}) and (\ref{<}) we infer that%
\begin{eqnarray}  \label{low}
\left\vert f\ast F^{\left( \mathbf{q}\right) }_{n\left( N_{k},x\right)
,1}\right\vert &\geq &\frac{\left\vert P_{N_{k}}\ast F^{\left( \mathbf{q}%
\right) }_{n\left( N_{k},x\right) ,1}\right\vert }{\sqrt[4]{\gamma _{j}}}%
-\left\vert \sum\limits_{j=1}^{k-1}\frac{\left( P_{N_{j}}\ast F^{\left(
\mathbf{q}\right) }_{n\left( N_{k},x\right) ,1}\right) }{\sqrt[4]{\gamma _{j}%
}}\right\vert  \label{ii555}\\
&\geq &c\sqrt[4]{\gamma _{k}}.  \notag
\end{eqnarray}
where, as before $x\in E^{\prime }$. }

{\normalsize Now, taking into account (\ref{O}), (\ref{O1}), (\ref{ae}) and (%
\ref{ii555}) one yields
\begin{equation*}
\sup\limits_{n}\left\vert T_{n}^{\left( \mathbf{q}\right) }(f;x)\right\vert
=\infty \text{ \ \ \ }
\end{equation*}%
almost everywhere. This completes the proof. }
\end{proof}

\begin{remark}
In Hardy's monograph \cite{hardy2000divergent}, the necessary and sufficient
conditions for the convergence of one sequence of operators, ensuring the
convergence of another sequence of operators, are explored. In particular,
it is established that the condition
\begin{equation}
\sup\limits_{n}\left( 2^{n}/Q_{2^{n}}\right) \leq c<\infty  \label{H}
\end{equation}%
ensures that the convergence of Fej\'er means implies the convergence of the
operators (\ref{seq}). 
Deriving the well-known theorem regarding the almost everywhere convergence
of Fej\'er means, we find that this condition guarantees (\ref{H}) the almost
everywhere convergence of the sequence of operators (\ref{seq}) for every
integrable function. By Theorem \ref{T1} we infer that the condition $\frac{%
2^{n}q_{2^{n}}}{Q_{2^{n}}}\geq c>0$ is necessary and sufficient for the
almost everywhere convergence of the sequence of operators $\{T_{n}^{\left(
\mathbf{q}\right) }\left( f\right)\}$ for all integrable functions.
\end{remark}

\section{Proof of Theorem \protect\ref{convmeas}}

\begin{proof}[Proof of Theorem \protect\ref{convmeas}.]
(i). Assume that $\beta \left( \mathbf{q}\right) +\beta \left( \mathbf{p}%
\right) >0$, then at least one either $\beta \left( \mathbf{q}\right) $ or $%
\beta \left( \mathbf{p}\right) $ is not equal to zero. Without loss of the
generality, we may assume that $\beta \left( \mathbf{p}\right) >0$.
Therefore, $\rho \left( \mathbf{p}\right) <\infty $ which yields $%
\sup\limits_{n}L_{n}^{\left( \mathbf{p}\right) }<\infty $. Hence, using %
\eqref{Lq}, we infer that the operator $T_{m}^{\left( \mathbf{p}\right) }$
is bounded from $L_{1}$ to $L_{1}.$ On the other hand, the uniform
boundedness of the sequence of operators $\{T_{n}^{\left( \mathbf{q}\right)
}\}$ from $L_{1}$ to $L_{1,\infty }$ has been established in \cite%
{goginava2023some}. Now, considering two-parametric operators obtained as a
result of the tensor product of one-dimensional ones, i.e. $T_{n}^{\left(
\mathbf{q}\right) }\otimes T_{n}^{\left( \mathbf{p}\right) }$, and employing
the well-known iteration method \cite[Ch. 17]{zygmund2002trigonometric}, we
can prove the following inequality $\left\Vert f\ast \left( T_{n}^{\left(
\mathbf{q}\right) }\otimes T_{n}^{\left( \mathbf{p}\right) }\right)
\right\Vert _{1,\infty }\leq c\left\Vert f\right\Vert _{1}$ and using the
standard method, we arrive at the desired assertion. \newline

(ii). Let $\beta \left( \mathbf{q}\right) +\beta \left( \mathbf{p}\right)=0$%
, then $\beta \left( \mathbf{q}\right)=\beta \left( \mathbf{p}\right)=0$.
This implies (iii) condition of Lemma \ref{QQ}. Therefore, by the proof of
Theorem \ref{T1}, one can find the sequences $\gamma_k(\mathbf{q})$ and $%
\gamma_k(\mathbf{p})$, associated with $\mathbf{q}$ and $\mathbf{p}$,
respectively. Notice that, these numbers enjoy the inequality \eqref{l1}.

Now, by employing the representations (\ref{O1}), (\ref{O2}), one finds
\begin{eqnarray}
T_{n,m}^{\left( \mathbf{q,p}\right) }\left( f,x,y\right) &=&f\ast \left(
F_{n,1}^{\left( \mathbf{q}\right) }\otimes F_{m,1}^{\left( \mathbf{p}\right)
}\right) \left( x,y\right)  \notag  \label{rep} \\
&&+f\ast \left( F_{n,1}^{\left( \mathbf{q}\right) }\otimes F_{m,2}^{\left(
\mathbf{p}\right) }\right) \left( x,y\right)  \notag \\
&&+f\ast \left( F_{n,2}^{\left( \mathbf{q}\right) }\otimes F_{m,1}^{\left(
\mathbf{p}\right) }\right) \left( x,y\right)  \notag \\
&&+f\ast \left( F_{n,2}^{\left( \mathbf{q}\right) }\otimes F_{m,2}^{\left(
\mathbf{p}\right) }\right) \left( x,y\right) .  \notag
\end{eqnarray}%
Following \cite{goginava2023some} we get
\begin{equation*}
\left\Vert f\ast F_{m,2}^{\left( \mathbf{p}\right) }\right\Vert _{1}\leq
c\left\Vert f\right\Vert _{1} \quad \left( f\in L_{1}\right) .
\end{equation*}%
Since $f\ast F_{m,1}^{\left( \mathbf{p}\right) }$ is martingale transform,
then by \cite{Burk} or \cite[Ch. 3]{SWS} one finds
\begin{equation*}
\left\Vert f\ast F_{m,1}^{\left( \mathbf{p}\right) }\right\Vert _{1,\infty
}\leq c\left\Vert f\right\Vert _{1}.
\end{equation*}%
Again, by the standard method of iteration of operators the following
inequalities can be established
\begin{equation}
\left\Vert f\ast \left( F_{n,2}^{\left( \mathbf{q}\right) }\otimes
F_{m,2}^{\left( \mathbf{p}\right) }\right) \right\Vert _{1}\leq c\left\Vert
f\right\Vert _{1},  \label{m1}
\end{equation}%
\begin{equation}
\left\Vert f\ast \left( F_{n,1}^{\left( \mathbf{q}\right) }\otimes
F_{m,2}^{\left( \mathbf{p}\right) }\right) \right\Vert _{1,\infty }\leq
c\left\Vert f\right\Vert _{1}  \label{m2}
\end{equation}%
and%
\begin{equation}
\left\Vert f\ast \left( F_{n,2}^{\left( \mathbf{q}\right) }\otimes
F_{m,1}^{\left( \mathbf{p}\right) }\right) \right\Vert _{1,\infty }\leq
c\left\Vert f\right\Vert _{1}.  \label{m3}
\end{equation}

Now, we consider the operator $f\ast F_{n,1}^{\left( \mathbf{q}\right)
}\otimes F_{n,1}^{\left( \mathbf{p}\right) }$. According to  \cite[Lemma 1]%
{GoginavaMeasureGMJ}, in order to prove Theorem \ref{convmeas}, it is enough
to show the existence of a sequence of functions $\{\xi _{k}\}_{k=1}^{\infty
}$ from the unit ball of $L_1(\mathbb{I}^2)$, the sequences of integers $%
\{n_{k}\}_{k=1}^{\infty }$ and an increasing sequence $\{\lambda
_{k}\}_{k=1}^{\infty }$ with $\lambda_k\to\infty$ such that%
\begin{equation*}
\inf_{k}\boldsymbol{\mu }\left\{\left( x,y\right) \in \mathbb{I}%
^{2}:\left\vert \xi _{k}\ast \left( F_{n_{k},1}^{\left( \mathbf{q}\right)
}\otimes F_{n_{k},1}^{\left( \mathbf{p}\right) }\right) \left( x,y\right)
\right\vert >\lambda _{k}\right\} >0.
\end{equation*}

For every $k\in\mathbb{N}$ we denote
\begin{equation*}
n_{k}:=2^{2k}+2^{2k-2}+\cdots +2^{2}+2^{0}
\end{equation*}%
and define%
\begin{equation*}
f_{k}\left( x,y\right) :=D_{2^{2k+1}}\left( x\right) D_{2^{2k+1}}\left(
y\right) .
\end{equation*}

Then, one can see that%
\begin{equation*}
F_{n_{k},1}^{\left( \mathbf{q}\right) }=\frac{w_{n_{k}}}{Q_{n_{k}}}%
\sum_{j=0}^{k}Q_{2^{2j}+2^{2j-2}+\cdots +2^{2}+2^{0}}w_{2^{2j}}D_{2^{2j}}.
\end{equation*}

So, we can write{\small
\begin{eqnarray}
\left\vert f_{k}\ast \left( F_{n_{k},1}^{\left( \mathbf{q}\right) }\otimes
F_{n_{k},1}^{\left( \mathbf{p}\right) }\right) \right\vert  &=&\left\vert
D_{2^{2K+1}}\ast \left( \frac{w_{n_{k}}}{Q_{n_{k}}}%
\sum_{j=0}^{k}Q_{2^{2j}+2^{2j-2}+\cdots
+2^{2}+2^{0}}w_{2^{2j}}D_{2^{2j}}\right) \right\vert   \label{ma} \\
&&\times \left\vert D_{2^{2k+1}}\ast \left( \frac{w_{n_{k}}}{P_{n_{k}}}%
\sum_{j=0}^{k}P_{2^{2j}+2^{2j-2}+\cdots
+2^{2}+2^{0}}w_{2^{2j}}D_{2^{2j}}\right) \right\vert   \notag \\
&=&\left\vert \frac{1}{Q_{n_{k}}}\sum_{j=0}^{k}Q_{2^{2j}+2^{2j-2}+\cdots
+2^{2}+2^{0}}w_{2^{2j}}D_{2^{2j}}\right\vert   \notag \\
&&\times \left\vert \frac{1}{P_{n_{k}}}\sum_{j=0}^{k}P_{2^{2j}+2^{2j-2}+%
\cdots +2^{2}+2^{0}}w_{2^{2j}}D_{2^{2j}}\right\vert .  \notag
\end{eqnarray}%
}{\normalsize We assume that $x\in I_{2l}\backslash I_{2l+1},l=0,1,\dots
,2k-1.$ Then from \eqref{Dir} it follows that }{\small
\begin{eqnarray}
\frac{1}{Q_{n_{k}}}\left\vert \sum_{j=0}^{k}Q_{2^{2j}+2^{2j-2}+\cdots
+2^{2}+2^{0}}w_{2^{2j}}D_{2^{2j}}\right\vert  &=&\frac{1}{Q_{n_{k}}}%
\left\vert \sum_{j=0}^{l}Q_{2^{2j}+2^{2j-2}+\cdots
+2^{2}+2^{0}}w_{2^{2j}}D_{2^{2j}}\right\vert   \notag  \label{211} \\
&=&\frac{1}{Q_{n_{k}}}\left\vert \sum_{j=0}^{l-1}Q_{2^{2j}+2^{2j-2}+\cdots
+2^{2}+2^{0}}2^{2j}-Q_{2^{2l}+2^{2j-2}+\cdots +2^{2}+2^{0}}2^{2l}\right\vert
.
\end{eqnarray}%
}{\normalsize Due to
\begin{equation*}
\sum_{j=0}^{l-1}Q_{2^{2j}+2^{2j-2}+\cdots +2^{2}+2^{0}}2^{2j}\leq \frac{%
2^{2l}}{3}Q_{2^{2\left( l-1\right) }+\cdots +2^{2}+2^{0}},
\end{equation*}%
right hand side of \eqref{211} can be estimated as follows
\begin{eqnarray}
RHS &\geq &\frac{1}{Q_{n_{k}}}\left( Q_{2^{2l}+2^{2j-2}+\cdots
+2^{2}+2^{0}}2^{2l}-\frac{2^{2l}}{3}Q_{2^{2\left( l-1\right) }+\cdots
+2^{2}+2^{0}}\right)   \label{2l} \\
&\geq &\frac{c2^{2l}Q_{2^{2l}}}{Q_{2^{2k}}}.  \notag
\end{eqnarray}%
Now, let $x\in I_{2l-1}\backslash I_{2l}$. Then }{\small
\begin{eqnarray}
\frac{1}{Q_{n_{k}}}\left\vert \sum_{j=0}^{k}Q_{2^{2j}+2^{2j-2}+\cdots
+2^{2}+2^{0}}w_{2^{2j}}D_{2^{2j}}\right\vert  &=&\frac{1}{Q_{n_{k}}}%
\sum_{j=0}^{l-1}2^{2j}Q_{2^{2j}+2^{2j-2}+\cdots +2^{2}+2^{0}}  \label{2l+1}
\\
&\geq &\frac{1}{Q_{n_{k}}}\sum_{j=0}^{l-1}2^{2l-2}Q_{2^{2\left( l-1\right)
}+\cdots +2^{2}+2^{0}}  \notag \\
&\geq &\frac{1}{4Q_{n_{k}}}2^{2l}Q_{2^{2\left( l-1\right) }}  \notag \\
&\geq &\frac{c2^{2l}Q_{2^{2l}}}{Q_{2^{2k}}}.  \notag
\end{eqnarray}%
}{\normalsize Hence, combining (\ref{2l}) and (\ref{2l+1}) yield
\begin{equation}
\frac{1}{Q_{n_{k}}}\left\vert \sum_{j=0}^{k}Q_{2^{2j}+2^{2j-2}+\cdots
+2^{2}+2^{0}}w_{2^{2j}}D_{2^{2j}}\right\vert \geq \frac{c2^{l}Q_{2^{l}}}{%
Q_{2^{2k}}},  \label{2112}
\end{equation}%
where $x\in I_{l}\backslash I_{l+1}$, $l=0,1,\dots ,2k-1.$ }

{\normalsize Let $\left( x,y\right) \in \left( I_{a}\backslash
I_{a+1}\right) \times \left( I_{b}\backslash I_{b+1}\right) $ for some $a\in
\left\{ 2k-\gamma _{k}\left( \mathbf{q}\right) ,\dots,2k-1\right\}$,\newline
$b\in \left\{ 2k-\gamma _{k}\left( \mathbf{p}\right) ,\dots,2k-1\right\} $.
Then, from (\ref{211}) and (\ref{2112}) we get
\begin{eqnarray}  \label{ab}
\left\vert f_{k}\ast \left( F_{n_{k},1}^{\left( q\right) }\otimes
F_{n_{k},1}^{\left( p\right) }\right) \right\vert &\geq &\frac{%
c2^{a}Q_{2^{a}}}{Q_{2^{2A}}}\frac{c2^{b}P_{2^{b}}}{Q_{2^{2k}}} \\
&\geq &c2^{a+b}.  \notag
\end{eqnarray}%
Without lost of generality, we may assume that $\gamma _{k}\left( \mathbf{q}%
\right) \geq \gamma _{k}\left( \mathbf{p}\right) $. Set $\lambda
_{k}:=4k-\gamma _{k}\left( \mathbf{q}\right) .$ }

{\normalsize Denote
\begin{equation*}
K_{k}:=\left\{ \left( x,y\right) :\left\vert f_{k}\ast \left(
F_{n_{k}}^{\left( \mathbf{q}\right) }\otimes F_{n_{k}}^{\left( \mathbf{p}%
\right) }\right) \left( x,y\right) \right\vert >2^{\lambda _{k}}\right\}.
\end{equation*}%
Then, we obtain
\begin{eqnarray}  \label{lo}
{\mu }(K_k)&\geq &c\sum\limits_{2k-\gamma _{k}\left(\mathbf{q}\right) \leq
a\leq 2k+\gamma _{k}\left( \mathbf{p}\right) -\gamma _{k}\left( \mathbf{q}%
\right) }\sum\limits_{b=\lambda _{k}-a}^{2k-1}\frac{1}{2^{a+b}} \\
&\geq &\frac{c\gamma _{k}\left( \mathbf{p}\right) }{2^{\lambda _{k}}}.
\notag
\end{eqnarray}%
}

{\normalsize \ Since \cite{goginava2019subsequences} there exists $\left(
x_{1},y_{1}\right) ,...,\left( x_{l\left( k\right) },y_{l\left( k\right)
}\right) \in \mathbb{I}^{2}$, $l\left( k\right) :=\left[ \frac{2^{\lambda
_{k}}/\gamma _{k}\left( \mathbf{p}\right) }{c}\right] +1,$ such that%
\begin{equation}
\left\vert \bigcup\limits_{j=1}^{l\left( k\right) }\left( K_{k}\dotplus
\left( x_{j},y_{j}\right) \right) \right\vert \geq c>0,  \label{1/2}
\end{equation}%
} {\normalsize Using the Stein's method \cite{stein1961limits} (see also (%
\cite[pp. 7-12]{garsia1970topics}), we can show that there exists $t_{0}\in
\mathbb{I}$, such that}{\small
\begin{equation*}
\boldsymbol{\mu }\left\{ \left( x,y\right) \in \mathbb{I}^{2}:\left\vert
\sum\limits_{j=1}^{l\left( k\right) }r_{j}\left( t_{0}\right) \left(
f_{k}\ast \left( F_{n_{k},1}^{\left( \mathbf{q}\right) }\otimes
F_{n_{k},1}^{\left( \mathbf{p}\right) }\right) \left( x\dotplus
x_{j},y\dotplus y_{j}\right) \right) \right\vert >2^{\lambda _{k}}\right\}
\geq \frac{1}{8}.
\end{equation*}%
Hence,
\begin{equation*}
\boldsymbol{\mu }\left\{ \left( x,y\right) \in \mathbb{I}^{2}:\frac{1}{%
l\left( k\right) }\left\vert \sum\limits_{j=1}^{l\left( k\right)
}r_{j}\left( t_{0}\right) \left( f_{k}\ast \left( F_{n_{k}}^{\left( \mathbf{q%
}\right) }\otimes F_{n_{k}}^{\left( \mathbf{p}\right) }\right) \left(
x\dotplus x_{j},y\dotplus y_{j}\right) \right) \right\vert >c\gamma
_{k}\left( \mathbf{p}\right) \right\} \geq \frac{1}{8}.
\end{equation*}%
}{\normalsize Set
\begin{equation*}
\xi _{k}\left( x,y\right) =\frac{1}{l\left( k\right) }\sum\limits_{j=1}^{l%
\left( k\right) }r_{j}\left( t_{0}\right) f_{k}\left( x\dotplus
x_{j},y\dotplus y_{j}\right) .
\end{equation*}%
Hence,%
\begin{equation*}
\boldsymbol{\mu }\left( \left\{ \left\vert \xi _{k}\ast \left(
F_{n_{k}}^{\left( \mathbf{q}\right) }\otimes F_{n_{k}}^{\left( \mathbf{p}%
\right) }\right) \right\vert >\gamma _{k}\left( \mathbf{p}\right) \right\}
\right) \geq \frac{1}{8}.
\end{equation*}%
It is easy to see that  }$\left\Vert \xi _{k}\right\Vert _{1}\leq 1.$%
{\normalsize \ Hence, this completes the proof. }
\end{proof}

\section{Examples}

In this section, we collect several nontrivial examples emerging from the
main results of this paper. Let us first examine examples related to Theorem \ref{T1}.

\begin{example}
$\left( C,\alpha \right) $-means: Let $q_{j}:=A_{j}^{\alpha -1},\alpha \in
(0,1]$, where
\begin{equation*}
A_{0}^{\alpha }=1,\,\,A_{n}^{\alpha }=\frac{\left( \alpha +1\right) \cdots
\left( \alpha +n\right) }{n!}.
\end{equation*}%
It is easy to see that $Q_{j}\sim j^{\alpha }$. Moreover,
\begin{equation*}
\frac{1}{Q_{2^{N}}}\sum\limits_{j=1}^{N}Q_{2^{j}}\leq \frac{c\left( \alpha
\right) }{2^{N\alpha }}\sum\limits_{j=1}^{N}2^{j\alpha }\leq c\left( \alpha
\right) <\infty
\end{equation*}%
Hence, $\left( C,\alpha \right) $-means $\left( \alpha >0\right) $ are
convergent almost everywhere for every $f\in L_{1}\left( \mathbb{I}\right) $
\cite{ya}.
\end{example}

\begin{example}
Let $q_{j}:=j^{\alpha -1},\alpha \in \lbrack 0,1)$. First, we consider the
case when $\alpha =0$. Then the sequence of operators coincides with the
logarithmic means
\begin{equation*}
T_{n}^{\left( \mathbf{q}\right) }(f;x):=\frac{1}{Q_{n}}\sum_{k=1}^{n-1}\frac{%
S_{k}(f;x)}{n-k}.
\end{equation*}
Note that $Q_{n}\sim \ln n$. Therefore,
\begin{equation*}
\sup\limits_{n\in \mathbb{N}}\frac{1}{Q_{n}}\sum_{k=1}^{|n|}Q_{2^{k}}\sim
\sup\limits_{n\in \mathbb{N}}\frac{\left\vert n\right\vert ^{2}}{\log \left(
n+1\right) }\sim \sup\limits_{n\in \mathbb{N}}\log \left( n+1\right) =\infty
,
\end{equation*}%
Hence, by Theorem \ref{T1} we infer the existence of a function $f\in L_{1}\left( \mathbb{I}\right) $, for which the
sequence $\{T_{n}^{\left( \mathbf{q}\right) }(f;x)\}$ diverges
almost everywhere \cite{GatGogiAMSDiv2}.

Now, let us assume that $\alpha \in (0,1)$. Then, one can see that%
\begin{equation*}
\frac{1}{Q_{2^{N}}}\sum\limits_{j=1}^{N}Q_{2^{j}}\leq \frac{c\left( \alpha
\right) }{2^{N\alpha }}\sum\limits_{j=1}^{N}2^{j\alpha }\leq c\left( \alpha
\right) <\infty
\end{equation*}
Hence, again Theorem \ref{T1} implies that the sequence $\{T_{n}^{\left( \mathbf{q}\right) }\left(
f\right) \}$ is convergent almost everywhere for every $f\in L_{1}\left(
\mathbb{I}\right)$ \cite{SimWeisz}.
\end{example}

\begin{example}
Let $q_{k}:=\frac{\log k}{k}$, then we have the following sequence of operators:
\begin{equation}
T_{n}^{\left( \mathbf{q}\right) }\left( f\right) :=\frac{1}{\log ^{2}\left(
n+1\right) }\sum\limits_{k=1}^{n-1}\frac{\log \left( n-k\right) S_{k}\left(
f,x\right) }{n-k}.  \label{lo}
\end{equation}%
One can examine that 
\begin{equation*}
Q_{2^{n}}\sim n^{2}
\end{equation*}%
and%
\begin{equation*}
\frac{1}{Q_{2^{n}}}\sum\limits_{k=1}^{n}Q_{2^{k}}\sim \frac{1}{n^{2}}%
\sum\limits_{k=1}^{n}k^{2}\rightarrow \infty \text{ }\left( k\rightarrow
\infty \right) ,
\end{equation*}%
Therefore, by Theorem \ref{T1}, there exists a function $f\in L_{1}\left( \mathbb{I}\right)
$, for which (\ref{lo}) unbondedly diverges almost everywhere.
\end{example}

Now, let us demonstrate several examples related to Theorem \ref{convmeas}.

\begin{example}
Assume that
\begin{equation*}
p_{j}=q_{j}:=\left\{
\begin{array}{c}
1,j=0 \\
0,j>0%
\end{array}%
.\right.
\end{equation*}%
Then%
\begin{eqnarray*}
T_{n,m}^{\left( \mathbf{q},\mathbf{p}\right) }(f;x) &:&=\frac{1}{Q_{n}}\frac{%
1}{P_{m}}\sum_{k=1}^{n}q_{n-k}p_{m-l}S_{k,l}(f;x,y) \\
&=&S_{n,m}(f;x,y)=f\ast \left( D_{n}\otimes D_{m}\right) \left( x,y\right) .
\end{eqnarray*}%
One can calculate that%
\begin{eqnarray*}
\beta \left( Q\right)  &=&\lim\limits_{k\rightarrow \infty }\frac{%
Q_{2^{k}}-Q_{2^{k-1}}}{Q_{2^{k-1}}}=0, \\
\beta \left( P\right)  &=&\lim\limits_{k\rightarrow \infty }\frac{%
P_{2^{k}}-P_{2^{k-1}}}{P_{2^{k-1}}}=0.
\end{eqnarray*}%
Then by Theorem \ref{convmeas} we infer that the set of the functions from $\mathit{L_{1}\left(
\mathbb{I}^{2}\right) }$ for which the sequence $S_{n,m}\left( f,x,y\right) $
is convergent in measure on $\mathbb{I}^{2}$ is the first Baire category in $%
L_{1}\left( \mathbb{I}^{2}\right) $.

We notice that the issues of convergence in measure for quadratic partial sums of
two-dimensional Walsh-Fourier series have been studied by Konyagin \cite%
{konyagin1993subsequence}, Getsadze \cite{getsadze1992divergence}, Lukomskii
\cite{luk1}.
\end{example}

\begin{example}
Let $q_{j}=p_{j}:=1/j$. Then
\begin{equation*}
T_{n,m}^{\left( \mathbf{q},\mathbf{p}\right) }(f;x,y):=\frac{1}{\log n\log m}%
\sum_{k=1}^{n-1}\sum\limits_{l=1}^{m-1}\frac{S_{k,l}(f;x,y)}{\left(
n-k\right) \left( m-l\right) }.
\end{equation*}%
It is easy to see that%
\begin{eqnarray*}
\beta \left( Q\right) &=&\lim\limits_{k\rightarrow \infty }\frac{%
Q_{2^{k}}-Q_{2^{k-1}}}{Q_{2^{k-1}}}=0, \\
\beta \left( P\right) &=&\lim\limits_{k\rightarrow \infty }\frac{%
P_{2^{k}}-P_{2^{k-1}}}{P_{2^{k-1}}}=0.
\end{eqnarray*}%
By Theorem \ref{convmeas} we conclude the set of the functions in $L_{1}\left( \mathbb{I}%
^{2}\right) $ for which the sequence $T_{n,m}^{\left( \mathbf{q},\mathbf{p}\right)
}\left( f;x,y\right) $ converges in measure on $\mathbb{I}^{2}$ is the first Baire category. 

This result has been proved by G\'at, Goginava and Tkebuchava \cite
{GoginavaMeasureGMJ}
\end{example}

\begin{example} Now, let us consider 
\begin{equation*}
q_{j}:=\left\{
\begin{array}{c}
1,j=0 \\
0,j>0%
\end{array}%
\right.
\end{equation*}%
and $p_{j}:=1/j$. One can see that 
\begin{equation*}
T_{n,m}^{\left( \mathbf{q},\mathbf{p}\right) }(f;x,y):=\frac{1}{\log m}%
\sum\limits_{l=1}^{m-1}\frac{S_{n,l}(f;x,y)}{\left( m-l\right) }.
\end{equation*}%
Then Theorem \ref{convmeas} yields that the set of the functions in $L_{1}\left( \mathbb{I}%
^{2}\right) $  for which the sequence $T_{n,m}^{\left( \mathbf{q},\mathbf{p}\right)
}\left( f;x,y\right) $ converges in measure on $\mathbb{I}^{2}$ is the first Baire category. 
\end{example}

%


\begin{example}
Let $q_{j}=p_{j}:=\frac{1}{\left( j+1\right) \log ^{\beta }\left( j+1\right)
},\beta <1$. Then 
$$Q_{n}=P_{n}\sim \log ^{1-\beta
}\left( n+1\right).$$ 
One can see that 
\small
\begin{eqnarray*}
T_{n,m}^{\left( \mathbf{q},\mathbf{p}\right) }\left( f;x,y\right) &=&\frac{1%
}{\log ^{1-\beta }\left( n+1\right) }\frac{1}{\log ^{1-\beta }\left(
m+1\right) }\sum_{k=1}^{n-1}\sum\limits_{l=1}^{m-1} 
\frac{S_{n-k,m-l}(f;x,y)}{\left( k+1\right) \log ^{\beta }\left( k+1\right)
\left( l+1\right) \log ^{\beta }\left( l+1\right) }.
\end{eqnarray*}%
\normalsize
Due to 
\begin{eqnarray*}
\beta \left( Q\right) &=&\lim\limits_{k\rightarrow \infty }\frac{k^{1-\beta
}-\left( k-1\right) ^{1-\beta }}{\left( k-1\right) ^{1-\beta }}=0, \\
\beta \left( P\right) &=&\lim\limits_{k\rightarrow \infty }\frac{k^{1-\beta
}-\left( k-1\right) ^{1-\beta }}{\left( k-1\right) ^{1-\beta }}=0.
\end{eqnarray*}%
Theorem \ref{convmeas} implies that the set of the functions in $L_{1}\left( \mathbb{I}%
^{2}\right) $ for which the sequence $T_{n,m}^{\left( \mathbf{q},\mathbf{p}\right)
}(f;x,y)$ converges in measure on $\mathbb{I}^{2}$ is the first Baire
category.
\end{example}

\section*{Declaration}{The author declare that they have no conflict of~interests.}

\end{document}